\documentclass[12pt]{amsproc}
\usepackage{amsaddr}
\usepackage{hyperref}
\newtheorem{theorem}{\sc Theorem}[section]
\newtheorem{lemma}[theorem]{\sc Lemma}
\newtheorem{proposition}[theorem]{\sc Proposition}
\newtheorem{corollary}[theorem]{\sc Corollary}

\newtheorem{remark}{\sc Remark}

\begin{document}

\title{Coprime commutators in profinite groups}

\author{Cristina Acciarri}
\address{\textnormal{Dipartimento di Scienze Fisiche, Informatiche e Matematiche, Universit\`a degli Studi di Modena e Reggio Emilia, Via Campi 213/b, I-41125 Modena, Italy }\\ \small \textit{Email address:} \textsf{cristina.acciarri@unimore.it}}

\author{Pavel Shumyatsky}
\address{\textnormal{Department of Mathematics, University of Brasilia, DF~70910-900, Brazil }\\ \small \textit{Email address:} \textsf{pavel@unb.br}}

\thanks{The first author is a member of ``National Group for Algebraic and Geometric Structures, and Their Applications'' (GNSAGA–INdAM). The second author  was partially  supported by the Conselho Nacional de Desenvolvimento Cient\'{\i}fico e Tecnol\'ogico (CNPq)  and Funda\c c\~ao de Apoio \`a Pesquisa do Distrito Federal (FAPDF), Brazil. Part of this work was done while the authors were visiting ICTP (Trieste). They thank ICTP-INdAM for the Collaborative and Research in Pairs grant and ICTP for excellent hospitality.
The authors thank the referees for helpful comments on an earlier version of the paper.}

\begin{abstract} 
By a coprime commutator in a profinite group $G$ we mean any element of the form $[x, y]$, where $x,y\in G$ and $(|x|,|y|)=1$. It is well-known that the subgroup generated by the coprime commutators of $G$ is precisely the pronilpotent residual $\gamma_\infty(G)$. There are several recent works showing that finiteness conditions on the set of coprime commutators have strong impact on the properties of $\gamma_\infty(G)$ and, more generally, on the structure of $G$. In this paper we show that if the set of coprime commutators of a profinite group $G$ is covered by countably many procyclic subgroups, then 
$\gamma_\infty(G)$ is finite-by-procyclic. In particular, it follows that $G$ is finite-by-pronilpotent-by-abelian.\end{abstract}

\maketitle

\section{Introduction} 

Let $G$ be a profinite group. A commutator of the form $[x,y]$, with $x,y\in G$ such that $(|x|,|y|)=1$, is called a coprime commutator of $G$. It is well-known (and easy to see) that the subgroup generated by the coprime commutators of $G$ is precisely the pronilpotent residual $\gamma_\infty(G)$, that is, the minimal normal subgroup $N\leq G$ with the property that $G/N$ is pronilpotent.

The present article continues the line of research showing that a subgroup of a profinite group must have a very restricted structure whenever a canonical set of generators of that subgroup is small. An important illustration for this phenomenon is provided by the results on conciseness of words (see \cite{acc23, AcSh14, DeMoSh23Isr, DeMoShEng1}). The results show that for many group-words $w$ the verbal subgroup $w(G)$ of a profinite group $G$ is finite whenever $w$ takes only finitely many values in $G$. Moreover, it is now known that for some words $w$ the verbal subgroup $w(G)$ of a profinite group $G$ is finite whenever $w$ takes less than continuously many values in $G$ (see \cite{HPS23, dks20, KhSh23, pishu}).

One source of motivation for the present work is the following theorem obtained in \cite{DMS_2021}: Suppose that a profinite group $G$ contains less than continuously many coprime commutators. Then $\gamma_\infty(G)$ is finite.

There are several papers showing that whenever $w$-values in a profinite group $G$ are covered by finitely, or countably, many subgroups with certain property the verbal subgroup $w(G)$ almost has that property as well (see \cite{AS_2017,AS_2017_2, AS_13, FAMS, Shu_16}). In particular, it was shown in \cite{AS_2017} that if the set of commutators in a profinite group $G$ is contained in union of countably many procyclic subgroups, then the commutator subgroup $G'$ is finite-by-procyclic.

The main result of this paper is as follows.

\begin{theorem}\label{main} Let $G$ be a profinite group and assume that the set of coprime commutators of $G$ is contained in a union of countably many procyclic subgroups. Then $\gamma_\infty(G)$ is finite-by-procyclic.
\end{theorem}

The theorem admits a converse: if $G$ is a profinite group such that $\gamma_\infty(G)$ is finite-by-procyclic, then the set of coprime commutators of $G$ is covered by countably many procyclic subgroups. This is because a profinite group is covered by countably many procyclic subgroups if and only if it is finite-by-procyclic (see \cite[Proposition 2.12]{AS_2017}).

Observe that if the set of coprime commutators of $G$ is contained in a union of only finitely many, say $m$, procyclic subgroups, then, in view of \cite{AS_2014}, there is a finite normal subgroup $N\leq\gamma_\infty(G)$ such that the order of $N$ is bounded in terms of $m$ only and $\gamma_\infty(G)/N$ is procyclic. However, in general it is not true that under the hypothesis of Theorem \ref{main} the set of coprime commutators of $G$ is covered by finitely many procyclic subgroups. 

Indeed, let $p$ be an odd prime, $C_p$ the cyclic group of order $p$ and $\Bbb Z_p$ the additive group of $p$-adic integers. Set $M=C_p\times \Bbb Z_p$ and observe that $M$ cannot be covered by finitely many procyclic subgroups while it is covered by countably many ones. Let $\alpha$ be the automorphism of $M$ such that $x^\alpha=x^{-1}$ for every $x\in M$ and let $G=M\langle\alpha\rangle$. Then every element of $M$ is a coprime commutator in $G$ and $\gamma_\infty(G)=M$. Thus, $G$ is as in Theorem \ref{main} while $\gamma_\infty(G)$ cannot be covered by finitely many procyclic subgroups.

The following corollary of Theorem \ref{main} follows shortly.

\begin{corollary}\label{cor} Let $G$ be a profinite group and assume that the set of coprime commutators of $G$ is contained in a union of countably many procyclic subgroups. Then $G$ is finite-by-pronilpotent-by-abelian.
\end{corollary}

Thus, a relatively mild assumption on the set of coprime commutators has strong implications for the structure of a profinite group.

\section{Preliminaries}
Our notation and terminology for profinite  groups is standard; see, for example,  \cite{rib-zal} and  \cite{wil}.  By a subgroup of a profinite group we always mean a closed subgroup. A subgroup (topologically) generated by a subset $S$ is denoted by $\langle S\rangle$. Throughout, homomorphisms of profinite groups are assumed to be continuous. Recall that centralizers are closed subgroups, while commutator subgroups $[B,A]=\langle [b,a]\mid b\in B,\;a\in A\rangle$ are the closures of the corresponding abstract commutator subgroups. As usual, the Frattini subgroup of a group $G$ is denoted by $\Phi(G)$.  

In the first lemma we  collect some standard facts about coprime actions, which are immediate from the corresponding results for finite groups; see for instance \cite[Lemmas 4.28 and 4.29]{Isaacs} and \cite[Theorems 5.2.3,  6.2.2(iv) and 5.1.4]{Gorenstein}.
\begin{lemma}\label{lem:standard}
    Let $A$ be a profinite group acting  on a profinite group~$G$ and suppose that $(|G|,|A|)=1$. Then
    \begin{enumerate}
        \item [\textup{(i)}] $G=[G,A]C_G(A)$ and if $G$ is abelian, then $G=[G,A]\times C_G(A)$;
          \item [\textup{(ii)}] $[G,A,A]=[G,A]$;
            \item [\textup{(iii)}] $C_{G/N}(A)=NC_G(A)/N$ for any $A$-invariant normal subgroup $N$ of $G$;
            \item[\textup{(iv)}]  If  $A$  acts faithfully on $G$, then $A$ acts faithfully on $G/\Phi(G)$ as well. 
    \end{enumerate}
\end{lemma}

\begin{lemma}\label{fatto1} Let $c$ be a positive integer and $G$ be a nilpotent group of class at most $c$. Let $x\in G$ be an element such that $x^n\in Z(G)$ for some positive integer $n$. Then $x$ commutes with $y^{n^{c-1}}$ for any $y\in G$.
\end{lemma}
\begin{proof}
Without loss of generality, given $y\in G$, we can assume that $G=\langle x,y\rangle$.  We argue by induction on the nilpotency class $c$ of $G$. If $c=1$, then $G$ is abelian and there is nothing to prove. Assume  $c\geq 2$ and let $z=y^{n^{c-2}}$. By inductive hypothesis,  the image of $z$ in the quotient $G/\gamma_c(G)$ commutes with the image of $x$, and so $[x,z]\in \gamma_c(G)$. Since $\gamma_c(G)$ is central in $G$, we deduce that $z\in Z_2(G)$. By combining this fact with the hypothesis that $x^n\in Z(G)$ we deduce that
$1=[x^n,z]=[x,z]^n=[x,z^n]=[x,y^{n^{c-1}}]$ and the result follows.\end{proof}

\begin{lemma}\label{fatto2} Let $G$ be a nilpotent profinite group. Suppose that $G/Z(G)$ is finite-by-procyclic. Then $G/Z(G)$ is finite.
\end{lemma}
\begin{proof}
Since $G/Z(G)$ is finite-by-procyclic, there are elements $a_1,\ldots,a_s, b$ in $G$  such that $G=Z(G)\langle a_1,\ldots,a_s, b\rangle$ and $a_1,\ldots,a_s$ have finite order modulo $Z(G)$. Therefore there exist positive integers $n_1,\ldots,n_s$ such that $a_i^{n_i}\in Z(G)$ for  $1\leq i\leq s$. Let $c$ be the nilpotency class of $G$. In view of Lemma~\ref{fatto1} we know that  $a_i$ commutes with $b^{n_i^{c-1}}$ for any $i$. Then the element $b^{(n_1\ldots n_s)^{c-1}}$ commutes with each $a_i$ and with $b$, so $b^{(n_1\ldots n_s)^{c-1}}\in Z(G)$.  Note that $a_1,\ldots,a_s$ and $b$ have all finite order modulo $Z(G)$.  Therefore  $G/Z(G)$ is finite because it is nilpotent and generated by finitely many elements of finite order. 
\end{proof}

\begin{lemma}\label{subgrchar}
Let $G$ be a finite-by-procyclic profinite group. Then $G$ has a finite characteristic subgroup $N$ such that $G/N$ is procyclic.
\end{lemma}
\begin{proof} If $G$ is finite, the result is obvious, so assume that $G$ is infinite. Since $G$ is finite-by-procyclic, there exists a finite normal subgroup $M$ in $G$ such that $G/M$ is procyclic. Hence $G'\leq M$ and it is finite as well. Since $G$ is finitely generated, the quotient $G/G'$ is a finitely generated abelian group and so the torsion elements of $G/G'$ form a finite subgroup. It follows that the torsion elements of $G$ form a finite subgroup, which we denote by $N$. Observe that $N$ is characteristic in $G$ and $G/N$ is procyclic. The proof is complete.
\end{proof}

For a profinite group $G$ we denote by $\pi(G)$ the set of prime divisors of the orders of finite continuous images of $G$. For an element $x\in G$ we write $\pi(x)$ to denote $\pi(\langle x\rangle)$. If $\pi(G)\subseteq\pi$, then we say that $G$ is a pro-$\pi$ group. Recall that Sylow theorems hold for $p$-Sylow subgroups of a profinite group (see, for example, \cite[Ch.\ 2]{wil}). If $H$ is a subgroup of $G$ such that $\pi(H)\subseteq\pi$, we say that $H$ is a pro-$\pi$ subgroup, or just a $\pi$-subgroup.

\begin{lemma}\label{abelian_fbp}
Let $G$ be an abelian profinite group and let $K$ be a Hall subgroup of $G$. If $K$ and $G/K$ are both finite-by-procyclic, then so is $G$.
\end{lemma}
\begin{proof}
Since $K$ is a Hall subgroup, $G$ contains a subgroup $L$ such that $G=K\times L$ and $(|K|,|L|)=1$. Clearly, $G$ contains a finite subgroup $N$ such that the images in $G/N$ of both $K$ and $L$ are procyclic. Since $(|K|,|L|)=1$, it follows that $G/N$ is a direct product of two procyclic subgroups of coprime orders. Therefore $G/N$ is procyclic and the lemma follows.
\end{proof}
 
 For a profinite group $G$, we recall that   $\gamma_\infty(G)$ is  the intersection of the terms of the lower central series of $G$.
It is clear that a finite group $G$ is nilpotent if and only if $\gamma_\infty(G)=1$. Therefore a profinite group $G$ is pronilpotent if and only if $\gamma_\infty(G)=1$. By a well-known property of finite groups $\gamma_\infty(G)$ is generated by all commutators $[x,y]$, where $x$ and $y$ have mutually coprime orders (see for example \cite[Theorem 2.1]{Pavel1}). We also recall  that the Fitting subgroup ~$F(G)$ of a profinite group $G$ is the maximal pronilpotent normal subgroup of~$G$.

\begin{lemma}\label{centralizer_in_theFitting}
Let $G$ be a profinite group. Then the centralizer of $\gamma_\infty(G)$ is contained in the Fitting subgroup $F(G)$.
\end{lemma}
\begin{proof}
Set $C=C_G(\gamma_\infty(G))$ and note that  $\gamma_\infty(G)\cap C$ is contained in $Z(C)$. Thus $\gamma_\infty(C/Z(C))=1$ and $C/Z(C)$ is pronilpotent. It follows that $C$ is pronilpotent too and so  is contained in $F(G)$. 
\end{proof}

\begin{lemma}\label{fbp_implies_pbf}
Let $G$ be  a profinite group and assume that $G$ is  finite-by-pronilpotent. Then $G$  is  pronilpotent-by-finite. \end{lemma}
\begin{proof}
Let $N$ be a finite normal subgroup of $G$ such that $G/N$ is pronilpotent. Observe that  the centralizer $C_G(N)$ is open  in $G$ since $G/C_G(N)$ acts as a permutation group on $N$. Moreover $C_G(N)$  is pronilpotent. The result follows.
\end{proof}

For an element $x$ of a group $G$, an Engel sink $E(x)$ of $x$ is a set such that for every $g\in G$ all sufficiently  long commutators $[g,x,x,\ldots, x]$ belong to $E(x)$, that is, for every $g\in G$ there is a positive integer $n(g,x)$ such that $[g,_{n} x]\in E(x)$ for all $n\geq n(g,x)$. 

\begin{lemma}\label{cyclic_sinks}Let $G$ be a finite group such that $G=PA$, where $P$ is a normal $p$-subgroup and $A$ is a nilpotent $p'$-subgroup. Assume also that for any $a\in A$ the subgroup $[P,a]$ is cyclic. Then every element of $G$ admits an Engel sink generating   a cyclic subgroup.
\end{lemma}
\begin{proof} We write $E(x)$ to denote the minimal Engel sink of an element $x\in G$. Moreover, if $K$ is a subgroup containing $x$ write  $E_K(x)$ for the minimal Engel sink of $x$ in $K$.  Since $G=PA$, for any element $x\in G$ we can write $x=ya$ for some $y\in P$ and $a\in A$. From the fact that $G/P$ is nilpotent, we deduce that $E(x)$  is contained in $P$. Moreover for any $x\in G$ we have
$$E(x)=E_{P\langle x\rangle}(x)=E_{P\langle a\rangle}(x).$$ Note that $[P,a]$ is normal in $P\langle a\rangle$, the quotient  $P\langle a\rangle/[P,a]$ is nilpotent  and  the image of $x$ in $P\langle a\rangle/[P,a]$ is an Engel element  in the quotient.  Thus $E_{P\langle a\rangle}(x)$ is contained in $[P,a]$. By hypothesis $[P,a]$ is cyclic and so the Engel sink $E(x)$ generates a cyclic subgroup, as desired. In particular for any $x\in G$, there is an element $a\in A$ such that $[P,a]$ is an Engel sink for $x$.
\end{proof}

Next lemma makes use of the following theorem \cite[Theorem 1.2]{AS_2019} about a finite group in which every element has an Engel sink generating a cyclic subgroup.
\begin{theorem}\label{AS19}
Let $G$ be a finite group in which every element admits an Engel sink generating a cyclic subgroup. Then $\gamma_\infty(G)$ is cyclic.
\end{theorem}

\begin{lemma}\label{procyclic_in HA}
Let $G=HA$, where $H$ is a normal abelian subgroup and $A$ is a pronilpotent subgroup such that $(|H|,|A|)=1$. If $[H,a]$ is procyclic for any $a\in A$, then $\gamma_\infty(G)=[H,A]$ is procyclic as well.
\end{lemma}
\begin{proof}
In any finite  continuous quotient $Q$ of $G$ every element $x$ of $Q$  admits an Engel sink  $E(x)$ generating a cyclic subgroup since a minimal sink of $x$ is contained in the image of $[H,a]$ for some $a\in A$.  In view of Theorem \ref{AS19} we know that  $\gamma_\infty(Q)$ is cyclic. Thus $\gamma_\infty(G)=[H,A]$ is procyclic, as desired.
\end{proof}

Recall that any prosoluble group $G$ has a Sylow basis, that is a family of pairwise permutable Sylow $p_i$-subgroups $P_i$ of $G$, exactly one for each prime, and any two Sylow bases are conjugate (see  \cite[Proposition~2.3.9]{rib-zal}). The
basis normalizer (also known as the system normalizer) of such a Sylow basis in $G$ is $T=\bigcap _i N_G(P_i)$. If $G$ is a prosoluble group and $T$ is a system normalizer in $G$, then $T$ is pronilpotent and $G =
\gamma _{\infty}(G)T$ (see \cite[Lemma~5.6]{rei}).

The next lemmas are well-known in the case of a finite group (cf. \cite[9.2.7]{Robinson_2} and \cite[Lemma 2.4]{AST_2014})
The profinite variations of those results are straightforward using the standard inverse limit argument.

\begin{lemma}\label{profinte_systemnormalizer}
Let $G$ be a prosoluble  group such that $\gamma_\infty(G)$ is abelian. If  $T$ is any system normalizer in $G$ then $\gamma_\infty(G) \cap T=1$ and $T$ is a complement of   $\gamma_\infty(G)$.
\end{lemma}
Recall that a profinite group $G$ is metapronilpotent if and only if $\gamma_\infty(G)$ is pronilpotent.
\begin{lemma}\label{profinite_Hall}
 Let $G$ be a  metapronilpotent group, $P$ a Sylow pro-$p$ subgroup of $\gamma_\infty(G)$ and $H$ a Hall pro-$p'$ subgroup of $G$. Then $P=[P,H]$.
\end{lemma}

\begin{lemma}\label{intersection_with_center}
 Let $G$ be a metapronilpotent group such that $\gamma_\infty(G)$ is abelian. Then  the intersection $\gamma_\infty(G)\cap Z(G)$ is trivial. 
\end{lemma}
\begin{proof}
Let $T$ be any system normalizer in $G$. Since  $\gamma_\infty(G)$ is abelian, by Lemma~\ref{profinte_systemnormalizer}, we have  $\gamma_\infty(G)\cap T=1$. The result follows from the fact that  $Z(G) \leq  T$. 
\end{proof}

We denote by $FC(G)$ the $FC$-centre of a group $G$, that is, the set of all elements $g\in G$ such that $|g^G|$ is finite. Note that if $G$ is a profinite group, then $FC(G)$ need not be closed so in this case we treat $FC(G)$ as merely a subset of $G$ rather than a subgroup. An element of a group $G$ belonging to $FC(G)$ is called an $FC$-element.

Next lemma makes use of the following theorem \cite[Theorem 1.3]{AS_2022} that deals with profinite groups whose $\pi$-elements have restricted centralizers, that is when the centralizers of $\pi$-elements are either finite or open.
\begin{theorem}\label{AS_2022} Let $\pi$ be a set of primes and $G$ a profinite group in which the centralizer of each $\pi$-element is either finite or open. Then $G$ has an open subgroup of the form $P\times Q$, where $P$ is an abelian pro-$\pi$ subgroup and $Q$ is a pro-$\pi'$ subgroup.
\end{theorem}
\begin{lemma}\label{commutator subgroup finite}
Let $A$ be a profinite group acting  on an abelian profinite group $M$ in such a way that $(|A|, |M|)=1$ and $[M,a]$ is finite for any $a\in A$. Then $[M,A]$ is finite as well. 
\end{lemma}
\begin{proof}Without loss of generality we can assume that $A$ acts  on $M$ faithfully. Let $B$ be an abelian subgroup  of $A$ and   $\pi=\pi(B)$. Note that every $\pi$-element of $MB$ is an $FC$-element. Indeed, by hypothesis, for any $\pi$-element $x\in MB$ we know that $[M,x]$ is finite, so the index $[M:C_M(x)]$ is finite and $[MB:C_{MB}(x)]$ is finite as well.   In view of Theorem \ref{AS_2022}  $MB$ has an open subgroup of the form $P\times Q$, where $P$ is an abelian pro-$\pi$ subgroup contained in $B$ and $Q$ is a pro-$\pi'$ subgroup. Observe that $Q\leq C_M(P)$ and so $C_M(P)$ is an open subgroup in $M$ and  normal in $MB$. Since $M$ is abelian, we have $M=C_M(P)\times [M,P]$ and therefore $[M,P]$ is finite. Since $P$ is open in $B$, we can write $B=\langle P, b_1,\ldots, b_s\rangle$ for some elements $b_i\in B$ and a positive integer $s$. Thus $[M,B]=[M,P]\prod_{i=1}^{s}[M,b_i]$ is finite, as desired.
 
By Lemma \ref{lem:standard}(ii) $B$ acts  on $[M,B]$ faithfully and so $B$ must be finite. We have shown that every abelian subgroup of $A$ is finite. Because of \cite[Theorem 2]{Zelmanov} this implies that $A$ is finite. Now the result is obvious since $[M,A]=\prod_{a\in A}[M,a]$. 
\end{proof}

\begin{lemma}\label{coverwithcs}
Let $G$ be a profinite group such that  the set of all  coprime commutators of $G$ is contained in a union of countably many procyclic subgroups, say $C_1,C_2,\ldots$. If $M$ is a subgroup of $G$ that entirely consists of coprime commutators, then there is a positive  index $j$ such that $M\cap C_j$ has finite index in $M$.
\end{lemma}
\begin{proof}
Since by hypothesis the set of all coprime commutators of $G$ is covered by the countably many procyclic subgroups $\{C_i\}_{i\geq 1}$, the subgroup $M$ is covered by the countably many subgroups $\{M\cap C_i\}_{i\geq 1}$. By Baire's Category Theorem \cite[p.\ 200]{Kelley} there exists an integer $j$, an element $a\in M$ and an open normal subgroup $N$ in $M$ such that $aN\subseteq M\cap C_j$. Thus $N$ is contained in $M\cap C_j$ and so $M\cap C_j$ has finite index in $M$, as desired.\end{proof}

We say that a subset $X$ of a group $G$ is normal  if it is invariant under the inner automorphisms of $G$.
\begin{lemma}\label{useful} Let $X$ be a normal subset of a profinite group $G$ and assume that $X$ is contained in a union of countably many procyclic subgroups. Let $x\in X$. Then the normalizer $N_G(\langle x\rangle)$ has finite index in $G$.
\end{lemma}

In the proof of the previous lemma we will use without explicit reference  the following fact: let $H$ be a subgroup of a profinite group $K$ and let $y$ be an element of $K$ such that $H^y\leq H$, then $H^y=H$. This is because if $H^y < H$, then the inequality would also hold in some finite image of $K$, which yields a contradiction.

\begin{proof}[Proof of Lemma~\ref{useful}] Let $C_1,C_2,\dots$ be countably many procyclic subgroups whose union contains $X$. Let $S_i$ be the set of all $g\in G$ such that $x^g\in C_i$, for $i\geq 1$. The sets $S_i$ are closed and cover $G$.  By the Baire Category Theorem there exists  an element $a\in G$, an open normal subgroup $T\leq G$, and an index $j$ such that $aT\subseteq S_j$. We claim that $T$ normalizes $\langle x\rangle$. Indeed, by the definition of $S_j$ we know that  the elements $x^{at}\in C_j$ for every $t\in T$. In particular $\langle x^a \rangle \leq C_j$ and so $x$ belongs to $(C_j)^{a^{-1}}$. Moreover $\langle x^{at}\rangle \leq C_j$ for every  $t\in T$ and so also  $\langle x^t \rangle =\langle x \rangle^t$ is contained in $(C_j)^{a^{-1}}$ for every $t\in T$. Since $(C_j)^{a^{-1}}$ is procyclic, we claim  that $\langle x \rangle^t=\langle x \rangle$ for every $t\in T$. For,  let $x_p$ be a generator of the Sylow pro-$p$ subgroup of $\langle x \rangle$;  note that both subgroups  $\langle x_p \rangle$ and $\langle x_p \rangle^t$ are contained in a procyclic pro-$p$ subgroup isomorphic to $\mathbb{Z}_p$ or $C_{p^k}$ for some positive integer $k$, because $(C_j)^{a^{-1}}$ has a unique Sylow pro-$p$ subgroup containing both of them. Since for any two subgroups $A$ and $B$ of $\mathbb{Z}_p$ or $C_{p^k}$ we have that either $A\leq B$ or $B\leq A$, we conclude that either $\langle x_p \rangle^t\leq \langle x_p \rangle$ or $\langle x_p \rangle\leq \langle x_p \rangle^t$. It follows that  $\langle x_p \rangle^t = \langle x_p \rangle$ for any $t\in T$ and for every $p\in \pi(\langle x\rangle)$. Hence,  $\langle x \rangle^t= \langle x \rangle$ for any $t\in T$, as claimed.  Therefore $T\leq N_G(\langle x \rangle)$  and, since $T$ is open in $G$, the result follows.
\end{proof}

In the proof of the next  lemma we use the well-known Schur's Theorem  that if  $G$ is a group whose center has finite index, then $G'$ is finite (\cite[4.12]{Robinson}). 
\begin{lemma}\label{use} Let $G$ be a pro-$p$ group such that $G=AB$, where $A$ and $B$ are two infinite procyclic normal subgroups. Then $G$ is abelian.
\end{lemma}
\begin{proof} Since $G$ is a pro-$p$ group, an infinite procyclic subgroup of $G$ is just infinite.  Note that both quotient $G/A$ and $G/B$ are abelian and so $G'\leq A\cap B$. In particular $G'$ is central in $G$. So if $G$ is nonabelian, then  $G'$ must be a subgroup of finite index in both factors $A$ and $B$. Since $G'\leq Z(G)$, it follows that $G$ is central-by-finite. In view of  Schur's Theorem   $G'$ is finite as well, which clearly is a contradiction since $G'$ is, for example, of finite index in $A$. We conclude that $G$ must be abelian and this completes the proof.
\end{proof}

\section{Main result}

Throughout  this section $G$ is a profinite group whose coprime commutators are contained in a union of countably many procyclic subgroups and  we want to show that $\gamma_\infty(G)$ is finite-by-procyclic. Note that our hypotheses are inherited by subgroups and quotients of the group $G$. 

Before proceeding with the proof of Theorem~\ref{main} we make a general observation that we will use throughout  this section.

\begin{remark}\label{Re1}
Let $K$ be a profinite group and $N$ a finite normal subgroup of $K$. Then $\gamma_\infty(K)$ is finite-by-procyclic if and only if  so is $\gamma_\infty(K/N)$.
\end{remark}

In what follows we will use the   following result that was already mentioned in the introduction (see \cite[Proposition 2.12]{AS_2017}). 
\begin{proposition}\label{AS17_prop}
 A profinite group $K$ is covered by countably many procyclic subgroups if and only if $K$ is finite-by-procyclic.
\end{proposition}

\begin{lemma}\label{000} Suppose $G=HA$, where $A$ is procyclic, $H$ is a normal subgroup, and $(|A|,|H|)=1$. Then $[H,A]$ is finite-by-procyclic.
\end{lemma}
\begin{proof} Let $A=\langle a\rangle$. So $[H,A]=[H,a]$ and without loss of generality, we can  assume that $H=[H,a]$, because $H$ is normal and $A$ acts coprimely on $H$. Observe that the result holds if $H$ is nilpotent. Indeed, suppose $H$ is nilpotent of class $c$. If $c=1$ and  $H$ is abelian, then every element of $H$ is a coprime commutator of the form $[h,a]$ for some $h\in H$. 
It follows from the hypothesis  that $H$ is covered by countably many procyclic subgroups and by Proposition \ref{AS17_prop} it is finite-by-procyclic, as desired. Assume now $c\geq 2$. 
The image of $H$ in the quotient  $G/Z(H)$ is   a nilpotent group of class less than $c$, so by  the inductive hypothesis,  we know that $H/Z(H)$ is finite-by-procyclic. Hence, by  Lemma~\ref{fatto2}, $H/Z(H)]$ is finite and, in view of the Schur theorem, the derived subgroup $H'$ is finite as well. Moreover observe that the quotient $H/H'$ is abelian, so by induction it follows that  $H/H'$ is finite-by-procyclic. Combining this with the fact that $H'$ is finite, we conclude that $H$ is finite-by-procyclic, as desired.

 Let us deal now with the situation where $H$ is not necessarily nilpotent. 
Let $C_1,C_2,\dots$ be countably many procyclic subgroups containing all coprime commutators of $G$. Set $S_i=\{x\in H \mid [x,a] \in C_i\}$ for every $i\geq 1$. The sets $S_i$ are closed and cover $H$. Hence by  Baire's Category Theorem \cite[p.\ 200]{Kelley}  there is an open  normal subgroup $N_0$ of $G$, an element $b\in H$ and a positive integer $j$ such that  $(N_0\cap H)b\subseteq S_j$. Set $N=N_0\cap H$. Observe that $N$ is $A$-invariant. We have  $[nb,a]\subseteq C_j$ for every $n\in N$. Using the formula $[xb,a]=[x,a]^b[b,a]\in C_j$, we deduce that $[x,a]^b$ belongs to $C_j$ for every $x\in N$. Thus $[x,a]\in (C_j)^{b^{-1}}$ for every $x\in N$ and so $[N,a]\subseteq  (C_j)^{b^{-1}}$ is procyclic. Set $C=[N,a]$.

 If $C$ is finite, then $a^N$ is finite and so $[N:C_N(a)]$ is finite as well. Combining this with the fact that $N$ is open in $H$ we deduce that $C_H(a)$ is open in $H$. We claim that $H$ is finite as well. Indeed,  let $L$ be an open normal subgroup of $H$ contained in $C_H(a)$. By \cite[Lemma 2.1(iv)]{AGS_2023} the subgroup $L$ is contained in $Z(H)$. Thus $Z(H)$ has finite index in $H$ and, by  Schur Theorem, $H'$ is finite.  Then we can assume that $H$ is abelian and $C=[N,a]$ becomes a finite normal subgroup of $H$. Passing to the quotient $H/C$ we can also assume that $N\leq C_H(a)$. Since $H=[H,a]$ is abelian we have $C_H(a)=1$. Therefore  $N$ is trivial and so $H$ is finite, as claimed.

It remains to deal with the case where $C=[N,a]$ is infinite. Note that $C$ is a normal subgroup  in $N$ since $N$ is $A$-invariant. Set $K=\langle C^H\rangle$. We now prove  that $H/K$ is finite. Indeed, the image of $N$ in the quotient $H/K=[H/K,a]$, say $\bar{N}$, is contained in $C_{H/K}(a)$.  By \cite[Lemma 2.1(iv)]{AGS_2023}  the image $\bar{N}$ of $N$ in this quotient is contained in $Z(H/K)$. Thus $Z(H/K)$ has finite index in $H/K$ and, again by  Schur Theorem, $(H/K)'$ is finite. By passing to the quotient over $(H/K)'$ we can assume that $H/K$ is abelian and $N$ becomes trivial in this quotient. Thus $N\leq K$ and it follows from the openness of $N$ in $H$ that $H/K$ is finite, as desired. 

 Let $P_1,\dots,P_t$ be  $a$-invariant Sylow $p_i$-subgroups of $H$ such that $H$ is generated by $K$ and $[P_i,a]$ for $1\leq i \leq t$ (cf. \cite[ p.\ 1341]{AGS_2023}).  Since $H/K$ is finite, we can choose finitely many $p$-elements (possibly over different primes $p$) $x_1,\dots,x_s$, each of the form $[g_i,a]$ for suitable $g_i\in P_i$  such that $H$ is generated by $K$ and $x_1,\dots,x_s$. In what follows we consider different cases depending on the properties of $C$.

{\it Case 1:} First  assume that $C$ is torsion-free and recall that $C$ is infinite procyclic. We will deduce from  Lemma~\ref{use} that  $K$ is abelian. Indeed, $K$ is nilpotent, being a product of finitely many normal procyclic subgroups. Moreover any Sylow pro-$p$ subgroup of $K$ is product of procyclic infinite subgroups that commute with each other by  Lemma~\ref{use}. Thus any Sylow pro-$p$ subgroup of $K$ is abelian and so is $K$.  Now we deduce from Lemma~\ref{useful}   that $H$ is central-by-finite.  Indeed, $K$ is abelian and we have  $H=K\langle x_1,\ldots,x_s\rangle$. In order to see that $H$ is central-by-finite, it is enough to show that $[K:C_K(x_i)]$ is finite, for any $1\leq i \leq s$, because, in view of that $T=\cap_{i=1}^{s}C_K(x_i)$ has finite index in $K$ and so, using that $H/K$ is finite, we get that $T$ has also finite index in $H$. Thus $T\leq Z(H)$ and $H$ is central-by-finite, as desired. Therefore it remains to show that for any $1\leq i \leq s$ the index $[K:C_K(x_i)]$ is finite. By Lemma~\ref{useful},  if $X$ is the set of all coprime commutators contained in $H$, then for any $i$ we know that $N_H(\langle x_i\rangle)$ has finite index in $H$. 
Therefore  $N_K(\langle x_i\rangle)$ has finite index in $K$ as well. Now set $U_i=N_K(\langle x_i\rangle)$ and consider $V_i=U_i\langle x_i\rangle$. Observe that $U_i$ is abelian since it is contained in $K$, so $V_i$ is nilpotent of class at most two and $\langle x_i\rangle$ is normal in $V_i$.  Note also that $[H:U_i]$ is finite and so  $[V_i:U_i] $ is finite too. Hence, there is a positive integer $m$ such that $x_i^m\in U_i$. It follows that  $x_i^m\in Z(V_i)$ because $U_i$ is abelian. By Lemma~\ref{fatto1} the element $x_i$ commutes with $y^m$ for every $y\in V_i$.  Therefore any element of  $U_i^m$ commutes with $x_i$ and so $U_i^m\leq Z(V_i)$, in particular  $U_i^m\leq C_K(x_i)$. Observe that $U_i$ is a finitely generated abelian group, since $K$ is finitely generated (being product of finitely many procyclic subgroups) and $U_i$ has finite index in $K$. Moreover also the quotient $U_i/U_i^m$ is finite since it is an abelian finitely generated group whose generators have finite order. We deduce that $U_i^m$ has finite index in $K$ so $[K:C_K(\langle x_i\rangle)]$ is finite, as desired.

 Since $H$ is central-by-finite, again by Schur Theorem, $H'$ is finite and so we can assume that $H$ is abelian.  Therefore $H=[H,a]$ consists of coprime commutators and we deduce   from Proposition \ref{AS17_prop}  that $[H,a]$ is finite-by-procyclic. Note that in this case $H$ is finite-by-(torsion-free procyclic).

{\it Case 2:} Suppose now that $C$ is infinite procyclic and  all Sylow subgroups of $C$ are finite. Then $K$ is nilpotent and all Sylow subgroups of $K$ are finite.  Since $K$ is open in $H$ all Sylow subgroups of $H$ are finite as well.
In view of  Lemma~\ref{useful}, we obtain that the elements $x_1,\dots,x_s$ belong to $FC(H)$. Therefore $K\cap C_H(x_1,\dots,x_s)$ is open in $H$. Thus  there is a subgroup $M\leq K\cap C_H(x_1,\dots,x_s)$ that is open and normal in $H$.  Note that  $H=\langle K, x_1,\ldots, x_s\rangle$  and the subgroup $M$ centralizes $x_i$ for every $i$.  Since $K$ is nilpotent, say of nilpotency class $c$, for any element $h\in H$ of the form $h=ky$, where  $k\in K$ and $y\in \langle x_1,\ldots, x_s\rangle$, we deduce that $[M,h]=[M, k]$. Moreover, for any choice of elements $h_1,\ldots,h_c\in H$, with $h_i=k_iy_i$ where $k_i\in K$ and $y_i\in \langle x_1,\ldots, x_s\rangle$, we have $[M,h_1,\ldots,h_c]=[M,k_1,\ldots,k_c]=1$. Hence $M\leq Z_c(H)$  and by Baer’s Theorem \cite[Corollary 2, p.\ 113]{Robinson} $\gamma_{c+1}(H)$ is finite. We can assume that $H$ is nilpotent. By the nilpotent case considered above we know that $H/\gamma_{c+1}(H)$ is finite-by-procyclic and so $H$ is finite-by-(procyclic with finite Sylow subgroups) as desired.

{\it Case 3:} Finally  we  deal with the general case, that is when  $C$ is  infinite procyclic with possibly infinite Sylow subgroups. Write $C=D_1\times D_2$, where $D_1$ is the product of all infinite Sylow subgroups of $C$ and $D_2$ is  the product of all finite Sylow subgroups of $C$. Set $K_1=\langle D_1^H\rangle$ and $K_2=\langle D_2^H\rangle$ respectively. Observe that  in the quotient  $H/K_1$ the image of $C$ has all Sylow subgroups finite, so by {\it Case 2} treated above, we know that $H/K_1$ is finite-by-(procyclic with all  finite Sylow subgroups) and, similarly applying {\it Case 1} to the quotient $H/K_2$, we obtain that  $H/K_2$ is finite-by-(torsion free procyclic).  Taking into account that $(|K_1|,|K_2|)=1$, the direct product $H/K_1\times H/K_2$ is finite-by-procyclic too and the same holds for $H$ since it is isomorphic to a subgroup of the direct product. This completes the proof.
\end{proof}

We shall be using the following  theorem  \cite[Theorem 1.1]{DMS_2021} that characterizes profinite groups containing less than continuously many coprime commutators. This was already mentioned in the introduction.
\begin{theorem}\label{DMS21} 
A profinite group $K$ is finite-by-pronilpotent if and only if the cardinality of the set of coprime commutators in $K$ is less than $2^{\aleph_0}$.
\end{theorem}

\begin{lemma}\label{2} Suppose $G=PA$, where $P$ is a normal pro-$p$ subgroup and $A$ is a pronilpotent $p'$-subgroup. Then $[A:C_A(P)]$ is finite.
\end{lemma}
\begin{proof} Note that the action of $A$ on $P$ is totally determined by the action of $A$ on $[P,A]$ because $C_A(P)=C_A([P,A])$ and  $[P,A]=[P,A,A]$. Therefore we can  assume that $P=[P,A]$. Moreover  in view of Lemma~\ref{lem:standard}(iv) we can pass to $P/\Phi(P)$ and assume that $P$ is elementary abelian. Note that all the coprime commutators of $G$ are contained in $P$ since $G=PA$ and $A$ is pronilpotent. From the assumption that all coprime commutators are contained in a union of countably many procyclic subgroups, we deduce that (by eventually taking the intersection of $P$ with the procyclic subgroups containing the coprime commutators) the coprime commutators of $G$ are contained in countably many subgroups of order $p$ since $P$ is elementary abelian. Thus  there are only countably many coprime commutators in $G$ and,  by Theorem \ref{DMS21}, the group $G$ is finite-by-pronilpotent. In particular $\gamma_\infty(G)$ is finite. Since $A$ is pronilpotent and $[P,A]=P$, we have $\gamma_\infty(G)=P$, so  $P$ is finite.  Taking into account that $A/C_A(P)$ acts faithfully on   $P$ the result follows.\end{proof}

\begin{lemma}\label{3} Suppose $G=PA$, where $P$ is a normal pro-$p$ subgroup and $A$ is a pronilpotent $p'$-subgroup. Then $[P,A]$ is finite-by-procyclic.
\end{lemma}

\begin{proof} As in the proof of the previous lemma, without loss of generality, we assume that $P=[P,A]$. In view of  Lemma~\ref{2}, by passing to the quotient $G/C_A(P)$, we can assume that $A$ is finite. For each $a\in A$, by Lemma~\ref{000}, the subgroup $[P,a]$ is finite-by-procyclic. By Lemma~\ref{subgrchar}, for any $a\in A$, the subgroup $[P,a]$ has a finite  characteristic subgroup $M_a$ such that $[P,a]/M_a$ is procyclic.  Then $M_a$ is a normal subgroup of $P$.  Let  $N_a=M_a^A$ be  the normal closure of $M_a$ in $G$. Observe that $N_a$ is  finite and  the image of $[P,a]$ in $G/N_a$ is procyclic. Set $N=\prod_{a\in A}N_a$ and pass to the quotient $G/N$. In this quotient  the image of $[P,a]$ is  procyclic for each $a\in A$. Therefore in any finite continuous homomorphic image $Q$ of $G/N$ we are in the position to apply Lemma~\ref{cyclic_sinks}  and obtain that any element of $Q$ admits an Engel sink generating a  cyclic subgroup. Then, in view of Theorem~\ref{AS19},  $\gamma_\infty(Q)$ is cyclic for any finite continuous homomorphic image $Q$ of $G/N$. Thus $\gamma_\infty(G/N)$ is procyclic and, since $N$ is finite, we conclude that $\gamma_\infty(G)=[P,A]$ is finite-by-procyclic, as desired.
\end{proof}

We introduce another  observation that will be useful in the subsequent arguments.
\begin{remark}\label{Re2} In a profinite group $K$, given $t\in K$ and  an open normal subgroup $N$ of $K$, any coset $tN$  contains an element $z$ such that $\pi(z)$ is finite. 
\end{remark} Indeed, set $\pi=\pi([K:N])$ and write an element $s\in tN$ in the form $s=zw$, where $z$ is a generator of the $\pi$-part of $\langle s\rangle$ and $w$ is a generator of the $\pi'$-part of $\langle s\rangle$. It follows that $w\in N$, so  $sN=zN$ and $\pi(z)$ is finite.

 For a set of primes $\pi$ we denote by $O_{\pi}(K)$ the unique largest normal pro-$\pi$ subgroup of  a profinite group $K$. Accordingly, if $K$ is pronilpotent we may write $K_\pi$ in place of $O_{\pi}(K)$.
\begin{proposition}\label{important} Suppose $G=HA$, where $H$ and $A$ are pronilpotent subgroups, $H$ is normal, and $(|A|,|H|)=1$. Then $[H,A]$ is finite-by-procyclic.
\end{proposition}
\begin{proof} Without loss of generality, we can assume that $H=[H,A]$. Note that if $A$ is procyclic, then the result follows from Lemma~\ref{000}, and if $H$ is a  pro-$p$ subgroup, then the result holds in view of Lemma~\ref{3}.  Therefore the Sylow subgroups of $H$  are finite-by-procyclic. We need to show that all but finitely many of them are procyclic. Since a pro-$p$ group is procyclic if and only if its quotient by the Frattini subgroup is cyclic, it follows that $H$ is finite-by-procyclic if and only if  $H/\Phi(H)$ is so.  Hence, by passing to the quotient $G/\Phi(H)$   we can assume that all Sylow subgroups of $H$ are elementary abelian. Moreover we can also assume that all Sylow subgroups of $H$ are finite.  Indeed, given a Sylow pro-$p$ subgroup $P$ of $H$, by Lemma~\ref{3} the Sylow subgroup $P=[P,A]$ is finite-by-procyclic and  so $P$ is elementary  abelian finitely generated and then finite.

 We need to show that  all but finitely many  Sylow subgroups of $H$  are cyclic. Suppose this is false and there are infinitely many noncyclic Sylow subgroups in $H$.  Let $M$ be the product of all cyclic Sylow subgroups of $H$. It is enough to prove that the image of $H$ in the quotient group $G/M$ if finite-by-procyclic because, then, by Lemma~\ref{abelian_fbp}, $H$ is finite-by-procyclic. So without loss of generality we may factor out $M$ and assume that all Sylow subgroups of $H$ are noncyclic.  Another fact that we will use is the following: when necessary, we can factor out  the product of  finitely many Sylow subgroups of $H$.
 
 For any $a\in A$, by Lemma~\ref{000} the subgroup $[H,a]$ is finite-by-procyclic. Since $H$ is pronilpotent and all Sylow subgroups of $H$ are abelian by assumption, we deduce that $H$ itself is abelian and so each subgroup $[H,a]$ is abelian and entirely consists  of coprime commutators. Let $C_1,C_2,\ldots$ the countably many  procyclic subgroups that contain all the coprime commutators of $G$. By Lemma~\ref{coverwithcs}, for any $a\in A$, there exists an index $j(a)$ such that $[H,a]\cap C_{j(a)}$ is open in $[H,a]$ for any $a\in A$. For any $a\in A$ set $m_a=[[H,a]: [H,a]\cap C_{j(a)}]$ the finite index of $[H,a]\cap C_{j(a)}$ in $[H,a]$, for any $a\in A$. Therefore there is a finite set of prime $\pi_a=\pi(m_a)$ such   that $[H_{\pi_a'},a]\subseteq C_{j(a)}$ for any $a\in A$. Consider now, for any $a\in A$, the pair $(\pi_a,j(a))$  as defined above. Note that the set $\{(\pi,j) \mid \pi \subset \mathbb{P}, |\pi|< \infty,  j\in \mathbb{N}^*\}$ is countable. Now for any pair $(\pi,j)$, where $\pi$ is a finite set of primes and $j$ is a positive integer, set 
 $$S_{(\pi,j)}=\{a\in A \mid \, [H_{\pi'},a]\subseteq C_j \}.$$  Let us show that the sets $S_{(\pi,j)}$ are closed. Indeed, if $b\notin S_{(\pi,j)}$, then there is $p\in \pi'$ such that $[H_{p},b]\not\subseteq C_j $. Thus there is an open normal subgroup $N$ of $G$ such that in $\overline{G}=G/N$ we have $[\overline{H_{p}},\bar{b}]\not\subseteq \overline{C_j} $. Then all elements of $Nb$ do not belong to $S_{(\pi,j)}$, that is, $A\setminus S_{(\pi,j)}$ is open and so $S_{(\pi,j)}$ is closed, as claimed.    The sets $S_{(\pi,j)}$ cover $A$, so in view of  Baire's Category Theorem   there exists an element $a\in A$, an open normal subgroup $B$ in $A$, an index $j$ and a finite set of primes $\pi$ such that $aB\subseteq S_{(\pi,j)}$ and so we have $[H_{\pi'},aB]\subseteq C_j$.  
 By factoring out the product of the finitely many  Sylow subgroups of $H$ corresponding to the primes in $\pi$, we can assume that $[H,aB]\subseteq C_j$ and that $[H,aB]$ is procyclic. Since $[H,a]$ and $[H,ab]$ are contained in $C_j$  for any $b\in B$, we claim that $[H,b]\subseteq [H,a][H,ab]\subseteq C_j$ for any $b\in B$.  Indeed, for any $b\in B$ we have  that the following holds: $[H,a^b]=[H,ab_1]$ where $b_1=[a,b]\in B$  and so $[H,a^b]\subseteq C_j$  for any $b\in B$. Observe that $[H,a]$ and $[H,a^b]$ are both cyclic subgroups of the same order and both are contained in $C_j$ that is procyclic. Therefore $[H,a]=[H,a^b]$ and we deduce that $[H,a]$ is $b$-invariant for any $b\in B$. In particular $[H,a]=[H,a^b]=[H,b^{-1}][H,ab]\subseteq C_j$ and so $[H,b]\subseteq[H,a][H,ab]\subseteq C_j$ for any $b\in B$, as desired. Hence $[H,B]$ is contained in $C_j$ and so it is a procyclic subgroup.

We can reduce to the case where $[C_H(B),A]$ is procyclic. Indeed, $H=[H,B]\times C_H(B)$ and  $\bar{A}=A/C_A(C_H(B))$ is finite. Moreover $[C_H(B),A]=[C_H(B),\bar{A}]=\prod_{\bar a \in \bar A}[C_H(B), \bar a]$. In view of Lemma~\ref{000} for any $\bar a\in \bar A$, we have that $[C_H(B),\bar a]$ is finite-by-procyclic  and since $\bar A$ is finite we can quotient out  the finite parts of all the $[C_H(B),\bar a]$ and assume that $[C_H(B), a]$ is procyclic for any $ a\in  A$.   By applying Lemma~\ref{procyclic_in HA} with  $C_H(B)A$ we deduce that $[C_H(B), A]=\gamma_\infty(C_H(B)A)$ is procyclic, as desired.

So $H$ is a direct product of two procyclic subgroups $[H,B]$ and $[C_H(B),A]$. Moreover,  we can assume that  $\pi([H,B])=\pi([C_H(B),A])$. For, if $\pi([H,B])$ and $\pi([C_H(B),A])$ have finitely many primes in common, we can factor out  the finite normal subgroup $M$ obtained as product of all the Sylow subgroups corresponding to the common primes and in the quotient $H/M$ all Sylow subgroups are procyclic and of coprime order, so $H/M$ is procyclic and then $H$ will be finite-by-procyclic, a contradiction. If  $\pi([H,B])$ and $\pi([C_H(B),A])$ have infinitely many primes in common, we can factor out  the product of all Sylow subgroups corresponding to the  primes that are not in the intersection $\pi([H,B])\cap \pi([C_H(B),A])$. Thus, without loss of generality, we assume that $\pi([H,B])= \pi([C_H(B),A]).$

Let $\pi_0$ be the set of primes $p$ such that the Sylow pro-$p$ subgroup of $A$ is contained in $B$ and set $B_0=O_{\pi_0}(A)$. Denote by $A_0$ the unique complement of $B_0$ in $A$. So $A=A_0\times B_0$, where $B_0\leq B$ and $\pi(A_0)$ is finite. 
By Lemma~\ref{commutator subgroup finite} if for any $a\in A_0$ the subgroup  $[C_H(B),a]$ is finite, then the subgroup $[C_H(B),A_0]$ is finite and $H$ would be finite-by-procyclic, a contradiction. Therefore there is $a\in A_0$ such that $[C_H(B),a]$ is infinite. 
Set $\sigma=\pi([C_H(B),a])$. Note that $G/O_{\sigma '}(H)$ is still a counterexample to the lemma, so we can assume that  $H=O_\sigma(H)$ and $\pi(H)=\sigma$. Observe also that $[C_H(B),a]=[C_H(B),A]$ because they are both procyclic, $[C_H(B),a]\leq [C_H(B),A]$ with $\pi([C_H(B),a])=\pi([C_H(B),A])$ and  all Sylow subgroups are cyclic of prime order, so they must be equal.
Moreover $[H,a]$ is finite-by-procyclic and so $[H,B,a]$ must be finite, otherwise one gets a contradiction since $[C_H(B),a]$ is procyclic. Therefore we can factor out  $[H,B,a]$ and assume that  $a$ centralizes $[H,B]$. If $[H,b]$ is finite for all $b\in B$, then again by Lemma~\ref{commutator subgroup finite} we  deduce that $[H,B]$ is finite, and get a contradiction. Otherwise choose $b\in B$ such that $[H,b]$ is infinite. Set $\tau=\pi([H,b])$. As done above we can now assume that  $H=O_\tau(H)$.  Recall that $H$ is the direct product of  $[H,B]$ and $[C_H(B),a]$ and we know that $a$ centralizes $[H,B]$. Then also $b$ centralizes $[H,a]$, since $b$ is acting trivially on $C_H(B)$ and $[H,a]=C_H(B)$. From the fact that $a$ centralizes $[H,b]$ and $b$ centralizes $[H,a]$, it follows that $[H,ab]$ contains both $[H,a]$ and $[H,b]$. Therefore $[H,ab]$ is not finite-by-procyclic and, in view of Lemma~\ref{000}, this is a final contradiction. This completes the proof.
\end{proof}

\begin{proposition}\label{super} Suppose $G$ is metapronilpotent. Then $\gamma_\infty(G)$ is finite-by-procyclic.
\end{proposition}
\begin{proof} Set $H=\gamma_\infty(G)$.  By passing to the quotient $G/\Phi(H)$   we can assume that all Sylow subgroups of $H$ are elementary abelian. Moreover we can also assume that all Sylow subgroups of $H$ are finite.  Indeed, given a Sylow $p$-subgroup $P$ of $H$, by Lemma~\ref{profinite_Hall} we have $P=[P,K]$, where $K$ is a Hall $p'$-subgroup of $G$. In view of Proposition~\ref{important} the Sylow subgroup $P=[P,K/C_K(P)]$ is finite-by-procyclic, so $P$ is elementary  abelian finitely generated and then finite.  
Once assumed that all Sylow subgroups of $H$ are elementary abelian and finite, we need to show that all but finitely many  of them are cyclic. Suppose this is false and there are infinitely many noncyclic Sylow subgroups in $H$.   As in the proof of Lemma~\ref{important},  without loss of generality, we can factor out  all the cyclic Sylow subgroups of $H$ and assume that all Sylow subgroups of $H$ are noncyclic.  
Moreover, we can factor out $O_{\pi}(H)$ whenever $\pi$ is a subset of $\pi(H)$ such that $\pi(H)\setminus \pi$ is infinite. This is because $G$ is a counterexample to the proposition if and only if $G/O_{\pi}(H)$ is. In particular, when necessary, we can also factor out the product of  finitely many Sylow subgroups of $H$.

 Let $A$ be any system normalizer in $G$. Since $H$ is abelian,  by Lemma~\ref{profinte_systemnormalizer} we have $G=HA$ and $A\cap H=1$. Note also that $A$ is pronilpotent (see \cite[Proposition 2.3.9]{rib-zal}). We also  have $H=[H,A]$, since $H=\gamma_\infty(G)$.  The centralizer $C_A(H)$ is a normal subgroup of $G$  and so $H$ is finite-by-procyclic if and only if the image of $H$ in the quotient $G/C_A(H)$ is finite-by-procyclic. Therefore, without loss of generality, by considering the quotient $G/C_A(H)$ we can assume that $A$ acts faithfully on $H$.

Set $$U=\{(x,y);\ x\in H,\ y\in A \text{ such that } (|x|,|y|)=1 \}.$$ Thus, $U$ is the subset of $H\times A$ that consists of pairs of elements of coprime orders. Moreover $U$ is closed in the profinite topology of $H\times A$ and it is covered by countably many closed subsets $S_i$, where $S_i=\{(x,y)\in U \mid [x,y]\in C_i\}$ and  the $C_i$'s are the countably many procyclic subgroups that cover all coprime commutators in $G$. By  Baire's Category Theorem at least one of those $S_i$ has non-empty interior. Thus  there exists a procyclic subgroup $C$, open normal subgroups $B\leq A$ and $J\leq H$, and elements $a_0\in A$ and $h_0\in H$ such that $(|h_0|,|a_0|)=1$ and $[x,y]\in C$ whenever $x\in h_0J$, $y\in a_0B$ with $(|x|,|y|)=1$.  Observe that $J$ can actually be taken characteristic  in $H$ because an open normal subgroup $M$ of $H$ contains all but finitely many  Sylow subgroups of $H$ and the product of those Sylow subgroups contained in $M$ is indeed a characteristic subgroup of $H$.  

 In view of Remark~\ref{Re2} we can assume, without loss of generality,  that $\pi(h_0)$ and  $\pi(a_0)$ are both finite.  By factoring out   finitely many Sylow subgroups of $H$,  corresponding to the finitely many prime divisors of $|a_0|$, we can also assume that $(|a_0|, |H|)=1$. Since $\pi(h_0)$ is finite, the subgroup $O_{\pi(h_0)}(H)$ is a finite characteristic subgroup, being the product of finitely many noncyclic Sylow subgroups of $H$ corresponding to the prime divisors of $|h_0|$. It follows that the centralizer $C_G(O_{\pi(h_0)}(H))$ is open in $G$, say of finite index $l$. Set now $B_0=O_{(\pi(a_0)\cup \pi(h_0)\cup \pi(l))'}(B)$ and observe that $\pi(A)\setminus \pi(B_0)$ is finite, since it is contained in $\pi(a_0)\cup \pi(h_0)\cup \pi(l)$. Moreover $B_0$ is a product of Sylow subgroups of $A$. 
 
 Write $A=A_0\times B_0$. By factoring out  the finitely many Sylow subgroups of $H$ corresponding to the primes in $\pi(A_0)$ we can reduce to the case where $\pi(H)\cap\pi(A_0)=\emptyset$.  Since $A_0$ and $H$ are pronilpotent and of coprime orders, Proposition~\ref{important} tell us that $\gamma_\infty(HA_0)=C_0$ is finite-by-procyclic.  By Remark~\ref{Re1} we can factor out $O_{\tau}(H)$, where $\tau$ is the set of primes dividing the order of the finite  part of $C_0$ and  assume that $C_0$ is procyclic. Note that the hypothesis that all Sylow subgroups of $H$ are noncyclic still holds.

 Next, we claim that whenever $x\in J$ and $y\in B_0$ with  $(|x|,|y|)=1$, the commutator $[x,y]$ belongs to $C$. Indeed, since $A$ and $H$ are both pronilpotent and, by the definition of $B_0$, we have $1=(|a_0|, |y|)=(|a_0|, |H|)=(|y|,|x|)=(|y|,|h_0|)$, we deduce that $\pi(a_0y)=\pi(a_0)\cup\pi(y)$ and $\pi(h_0x)\subseteq \pi(h_0)\cup \pi(x)$. Then $(|a_0y|, |h_0x|)=1$. Note that $[a_0y,h_0x]\in C$ and $[a_0y,h_0x]=[a_0,x]^y[y,x][a_0,h_0]^y[y,h_0]$.   Now $[a_0,x]\in C$, since $[a_0,h_0x]$ and $[a_0,h_0]$ both are in $C$.  Since $J$ is characteristic in $H$ and $A$ acts on $H$, we also obtain that $[a_0,x]^y\in C$. Finally since $[a_0,h_0]\in C$ and $B_0$ centralizes $h_0$ (by the definition of $B_0$), we conclude that $[a_0,h_0]^y\in C$  and $[y,h_0]=1$, therefore  $[y,x]$ belongs to $C$ whenever $x\in J$, $y\in B_0$ and $(|x|,|y|)=1$, as claimed.  
 
It follows that $\gamma_\infty(JB_0)\leq C$. Indeed, since $JB_0$ is metapronilpotent, in view of Lemma \ref{profinite_Hall}  $\gamma_\infty(JB_0)$ is generated by  commutators of the form $[b,j]$ with $j$ is a $p$-element of $J$ and $b$ is a $p'$-element of $B_0$. Since $J$ has finite index in $H$, by factoring out  the finitely many Sylow subgroups of $H$ corresponding to the prime divisors of $[H:J]$, we can now assume that $\gamma_\infty(HB_0)\leq C$ and set $C_1=\gamma_\infty(HB_0)$. Note that $H=C_0C_1$. If $C_1$ is finite, then $\gamma_\infty(G)$ is finite-by-procyclic and we are done. Assume that $C_1$ is infinite and  write $C=C_1\times C^{0}$. Set $\pi=\pi(C^0)$. Since $\pi(H)\setminus \pi$ is infinite we can  factor out  $O_{\pi}(H)$  assume that $C_1=C$.

 It follows that $H=C_0\times C$. Moreover $\pi(C)=\pi(C_0)$ because if one prime $p\in \pi(C)$ and does not belong to $\pi(C_0)$, then the corresponding Sylow subgroup of $H$ would be cyclic and this leads to a contradiction since we have assumed that all Sylow subgroups of $H$ are noncyclic. 
Furthermore we claim that  $C\leq C_H(A_0)$  and  $C_0\leq C_H(B_0)$. Indeed, on the one hand $A_0$ acts on $C_0\times C=C_H(A_0)[H,A_0]$, so $C=C_H(A_0)$. On the other hand $B_0$ acts on $C_0=[H,A_0]$ and  $\gamma_\infty(C_0B_0)=1$ since $$C=\gamma_\infty(HB_0)=[H,B_0]=[C_0\times C, B_0]=\gamma_\infty(C_0B_0)\times C.$$ Therefore $C_0B_0$ is pronilpotent and $C_0$ is central in $C_0B_0$ being the cartesian product of normal cyclic subgroups of prime order. In particular $C_0\leq C_H(B_0)$, as claimed. 

With these assumptions $A_0\cap B$ centralizes $H$. Indeed, if $x\in A_0\cap B$,  as $(|a_0x|, |H|)=1$, then  we have $[H,a_0x]\leq C$ and $[H,a_0x]\leq C_0$, so for any $x\in A_0\cap B$ the element $a_0x$ acts trivially on $H$ because $C\cap C_0=1$. In particular $a_0$ acts trivially on $H$ and so does $x$, for any $x\in A_0\cap B$. Note that $B=(B\cap A_0)\times B_0$ and from the fact that $B\cap A_0$ centralizes $H$ it follows that $[H,B]=[H,B_0]$. Since $A$ acts faithfully on $H$, we have $ A_0\cap B=1$. In particular $A_0$ is  finite because $B$ is open in $A$. 

Since $\gamma_\infty(HB_0)=\gamma_\infty(HB)$ is procyclic, $A$ has an open normal subgroup $B^{*}$ which  is maximal with respect to the property that $B\leq B^{*}$ and $\gamma_\infty(HB^{*})$ is procyclic. Let us summarize what we have:  $H=C\times C_0$ and $\pi(H)=\pi(C)=\pi(C_0)$ with $C=\gamma_\infty(HB)=\gamma_\infty(HB^*)$. Next we argue by induction on the  index $r$ of  $B^{*}$ in $A$. If $r=1$, then $B^{*}=A$ and there is nothing to prove. 

If $r$ is prime, then we have $A=\langle a \rangle B^{*}$. Since $B\leq B^*$, the element $a$ can be chosen inside $A_0$. Remark that $\gamma_\infty(HB^*)=C$ and $\gamma_\infty(HA_0)=C_0$.   By the above argument, $A_0\cap B^*=1$ and so $A_0=\langle a\rangle$ while $B^*=B$.
Set $K=aB$ and  $$W=\{(x,y);\ x\in H,\ y\in K \text{ such that } (|x|,|y|)=1 \}.$$ Thus, $W$ is the subset of $H\times K$ that consists of pairs of elements of coprime orders. Note that $W$ is closed  and  is covered by countably many closed subsets $T_i$, where $T_i=\{(x,y)\in W \mid [x,y]\in C_i\}$, where  the $C_i$'s are the countably many procyclic subgroups that cover all coprime commutators in $G$. Again by  Baire's Category Theorem at least one of those $T_i$ has non-empty interior.  Thus  there exists a procyclic subgroup $C_2$, open normal subgroups $B_1\leq B$  and $H_1\leq H$, and elements $b_0\in B$ and $h_1\in H$ such that  $[x,y]\in C_2$ whenever $x\in h_1H_1$, $y\in ab_0B_1$ with $(|x|,|y|)=1$. As above, in view of Remark~\ref{Re2} we can assume that $\pi(ab_0)$ is finite and by factoring out  the product of the finitely many Sylow subgroups of $H$  corresponding to the  prime divisors of $|ab_0|$ and of the index of $H_1$ in $H$, we can also assume that $H=H_1$, the element $h_1=1$ and $(|ab_0|, |H|)=1$. 
Next we assume that $[H,a]$ is infinite, since if $[H,a]$ were finite, then $[H,A_0]=C_0$ would be finite and so $H$ would be finite-by-procyclic, a contradiction. Set $\tau_0=\pi(H)\setminus \pi([H,a])$. By  factoring out  $O_{\tau_0}(H)$ we can also assume that $ \pi(H)=\pi([H,a])$.   If $[H,b_0]$ is infinite, then we get a contradiction because $[H,b_0]\leq C_0$ and $\pi([H,b_0])\subseteq \pi([H,a])=\pi(H)$, so $[H,ab_0]=[H,b_0]^a[H,a]\leq CC_0$ would be non-procyclic but at the same time $[H,ab_0]\leq C_2$ would  also be procyclic. Assume now that $[H,b_0]$ is finite. By factoring out  $[H,b_0]$ we can also assume that $[H,b_0]=1$. Then, on the one hand we have $[H,ab_0]\leq C_2$; on the other hand we also have $[H,ab_0]=[H,a]\leq C_0$. Since $H=C_0\times C$  and $\pi(C_0)=\pi(C)$, we deduce that $C_2=C_0$.  Therefore for any $b\in B_1$ we have $[H,ab_0b]=[H,ab]=[H,b][H,a]\leq C_0$. Since also $[H,a]\leq C_0$, we deduce that $[H,b] \leq C_0$ for any $b\in B_1$. Since for any $b\in B_1$ the commutator subgroup $[H,b]$ is also contained in $C$  and $C_0\cap C=1$, we have $[H,B_1]=1$. It follows that $[H,B]=[H,B/B_1]$ and also $[H,A]=[H,A/B_1]$. Set $\omega=\pi(A/B_1)$  and  by factoring out  $O_{\omega}(H)$ we can assume that $(|H|,|A|)=1$. In view of Lemma~\ref{important}, the subgroup $H$ is finite-by-procyclic, a contradiction.

Assume now that $r$ is not a prime. Then $A=A_0B^{*}$ and $|A_0|$ is not a prime. Since $A_0$ is a finite nilpotent group, let $A_1$ be a normal subgroup of prime index in $A_0$ and let $B_2=A_1\times B$. Note that $B_2$ is normal in $A$ and $B_2$ acts on $H$.  We apply the inductive hypothesis to the action of $B_2$ on $H$, since the index of $B^{*}$ in $B_2$ is smaller than the index of $B^{*}$ in $A$, and deduce that $\gamma_\infty(HB_2)$ is finite-by-procyclic. In view of Remark~\ref{Re2} we can also assume that $\gamma_\infty(HB_2)$ is procyclic. Since $[H,B_0]=C$  is a maximal procyclic subgroup of $H$ and $\gamma_\infty(HB_0)\leq \gamma_\infty(HB_2)$, we deduce that $[H,B_2]=C$. Hence, using the previous notation, we conclude that $B_2=B^*$. Taking into account that $[A:B_2]$ is a prime, it follows that $r$ is a prime. In view of the above, the proof is now complete. 

\end{proof}

Recall that a group $K$ possesses a certain property virtually if it has an open subgroup with that property.
\begin{lemma}\label{vpronilpotent} Suppose $G$ is virtually pronilpotent. Then $\gamma_\infty(G)$ is finite-by-procyclic.
\end{lemma}
\begin{proof} Let $F=F(G)$. We argue by induction on the index  of $F$ in $G$.  If  $G$ is pronilpotent  there is nothing to prove.
Assume that $[G:F]$ is at least two. 
 
Suppose first that  $G/F$ has a proper normal subgroup $H/F$. By induction, $\gamma_\infty(H)$ is finite-by-procyclic. In view of Remark~\ref{Re1}  we can assume that $\gamma_\infty(H)$ is procyclic. Set  $K=\gamma_\infty(G)$. By induction $K/\gamma_\infty(H)$ is finite-by-procyclic. So $G$ contains a normal subgroup $N$ such that  $\gamma_\infty(H)\leq N\leq  K$ and  $N/\gamma_\infty(H)$ is finite while $K/N$ is procyclic. Note that $\gamma_\infty(H)$ is normal in $G$ and so $G/C_G(\gamma_\infty(H))$ is abelian. In particular   $\gamma_\infty(H)\leq Z(G')$. It follows that  $[N:Z(N)]$ is finite and, by Schur Theorem, that $N'$ is finite.
 In view of Remark~\ref{Re1} we can factor out $N'$ and assume that $N$ is abelian, so  $K$ is metabelian. Hence $\gamma_\infty(K)$ is  contained in $N$.  In view of Lemma~\ref{intersection_with_center} we have that  $\gamma_\infty(K)\cap Z(K)=1$ and so  $\gamma_\infty(K)$ is finite. Again in view of Remark~\ref{Re1} we can factor out $\gamma_\infty(K)$ and assume that $K= \gamma_\infty(G)$ is pronilpotent. Thus $G$ is metapronilpotent and, by Lemma~\ref{super}, the subgroup  $\gamma_\infty(G)$ is finite-by-procyclic, as desired.

Therefore it remains to deal with the case where $\bar G=G/F$ is simple. If $\bar G$ is abelian, then $G$ is metapronilpotent  and the result follows from Lemma~\ref{super}.   We assume now that $\bar G$ is nonabelian.  In view of \cite[Theorem 7.4.5]{Gorenstein}, we know that for any prime $q\in\pi(\bar G)$ there is a nontrivial $q$-subgroup $\bar Q\leq\bar G$ and a $q'$-element $\bar a\in\bar G$ such that $\bar Q=[\bar Q,\bar a]$.

Choose $q\in\pi(\bar G)$, a $q$-subgroup $\bar Q\leq\bar G$ and an element $\bar a\in\bar G$ as above and suppose first that $F$ is a pro-$q'$ group. Let $T$ be the full preimage of $\bar Q\langle\bar a\rangle$ in $G$. Since $T$ is prosoluble,  $T/F(T)$ is finite soluble. 
From the case discussed above where $G/F$ was not simple, we deduce that $\gamma_\infty(T)$ is finite-by-procyclic. By Lemma \ref{subgrchar} $\gamma_\infty(T)$ contains a finite characteristic subgroup $M_0$ such that $\gamma_\infty(T)/M_0$ is procyclic.  Then $M=M_0^G$ is a finite normal subgroup of $G$ such that $\gamma_\infty(T)M/M$ is procyclic.   
We quotient out  $M$ and assume that $\gamma_\infty(T)$ is procyclic. Let $Q$ be a Sylow $q$-subgroup of $T$. Since the group of automorphisms of a procyclic group is abelian,  $T'$ centralizes $\gamma_\infty(T)$. Moreover $Q=[Q,a]$, where $\langle a \rangle F$ is the full preimage of $\langle \bar a \rangle$ in $G$. In particular $Q$ is contained in $T'$ and so   $Q$ centralizes $\gamma_\infty(T)$.   By Lemma~\ref{centralizer_in_theFitting}, since $T$  is procyclic-by-pronilpotent, we know that $C_T(\gamma_\infty(T))$ is contained in $F(T)$. Thus we deduce that $Q\leq F(T)$. Keeping in mind that $F$ is a pro-$q'$ group, we deduce that  $Q$ centralizes $F$.  Since $F$ is normal in $G$ and $Q\leq C_G(F)$, we obtain also that $[G,Q]$ centralizes $F$. Therefore $[G,Q]$ is central-by-finite and, by Schur Theorem,  $[G,Q]'$ is finite. We pass to the quotient $G/[G,Q]'$ and assume that $[G,Q]$ is abelian. Then $\bar G=[\bar G,\bar Q]$ and so $G=[G,Q]F$ is prosoluble. Note that if $F$ is a pro-$q'$ group, then the Sylow $q$-subgroup of $\gamma_\infty(G)$ is finite and therefore,  again by the case discussed above, we obtain that $\gamma_\infty(G)$ is finite-by-procyclic.

Finally we   drop the assumption that $F$ is a pro-$q'$ group. Write $F=Q_0\times F_0$, where $Q_0$ is the (unique) Sylow $q$-subgroup of $F$ and $F_0=O_q'(F)$. In view of the  above the image of $\gamma_\infty(G)$ in $G/Q_0$ is finite-by-procyclic. The same is true with regard to the image of $\gamma_\infty(G)$ in $G/F_0$: indeed, if $F$ is a pro-$q$ group, we choose a prime $q\neq p\in\pi(\bar G)$, a nontrivial $p$-subgroup $\bar P\leq\bar G$ and a $p'$-element $\bar b\in\bar G$ such that $\bar P=[\bar P,\bar b]$ and  repeat the argument of the previous paragraph. Thus, the images of $\gamma_\infty(G)$, respectively  in $G/F_0$ and $G/Q_0$, both are finite-by-procyclic with the respective procyclic sections of coprime orders. The result follows.
\end{proof}

\begin{lemma}\label{metapronilpotent} The group $G$ is  pronilpotent-by-finite-by-pronilpotent.
\end{lemma}
\begin{proof}   Given any coprime commutator $x\in G$, by Lemma~\ref{useful}  the subgroup $N_G(\langle x\rangle)$ has finite index in $G$. Consider now an open normal subgroup $K$ contained in $N_G(\langle x\rangle)$ and let  $e=e(x)=[N_G(\langle x\rangle):K]$. Then  $x^e\in K$. Observe that for any $g\in G$ the procyclic subgroup $\langle (x^e)^g\rangle$ is normal in $K$ and all the procyclic subgroups of such type normalize each other. Now let $M(x)$ be the subgroup generated by all those normal procyclic subgroups of type $\langle (x^e)^g\rangle$, for  $g\in G$. Since a subgroup generated by normal procyclic subgroups is pronilpotent, it follows that $M(x)$ is a normal pronilpotent subgroup.  So the subgroup $N=\prod_{x\in G}M(x)$ is pronilpotent as well.  In the quotient $G/N$ any image of  a coprime commutator has finite order. Therefore $G/N$ has only countably many coprime commutators and so, by Theorem \ref{DMS21}, $G/N$ is finite-by-pronilpotent. The result follows.
\end{proof}

\begin{lemma}\label{prosolu} 
The group $G$ is finite-by-prosoluble.
\end{lemma}

\begin{proof}  First we deal with  the case where $G$ is virtually pronilpotent.  Lemma~\ref{vpronilpotent} tells us that  $\gamma_{\infty}(G)$ is finite-by-procyclic and so $G$ is finite-by-prosoluble, as desired.  In the general case, by Lemma~\ref{metapronilpotent} the group $G$ is pronilpotent-by-finite-by-pronilpotent.  We deduce from  the previous  case that $G$ is finite-by-prosoluble-by-pronilpotent and the result follows.\end{proof}

Now we are ready to complete the proof of Theorem \ref{main} which we restate here for the reader's convenience. We write $F_2(G)$ to denote the second term of the upper Fitting series of $G$.
\begin{theorem}\label{last} 
The group $G$ is finite-by-procyclic-by-pronilpotent.
\end{theorem}

\begin{proof}  In view of Lemma~\ref{prosolu} the group $G$ is finite-by-prosoluble  and by Remark~\ref{Re1} it is sufficient to deal with the case  where $G$ is prosoluble.  By Lemma~\ref{fbp_implies_pbf} a profinite group, that is  finite-by-pronilpotent,  is actually pronilpotent-by-finite. Therefore  Lemma~\ref{metapronilpotent} implies that $G$ is virtually metapronilpotent. Let $j$ be the minimal number such that $G$ has a metapronilpotent normal subgroup of index $j$. We will argue by induction on $j$. If  $j=1$, then $G$ is metapronilpotent and the theorem holds by Proposition~\ref{super}.   

Assume now that $j\geq2$ and use induction on $j$.  Suppose first that $G/F_2(G)$ has a proper normal subgroup $H/F_2(G)$. By induction $\gamma_\infty(H)$ is finite-by-procyclic.  Since the finite part of $\gamma_\infty(H)$ is characteristic in $G$, in view of Remark~\ref{Re1} we can factor out  it and assume that $\gamma_\infty(H)$ is procyclic. As observed  in the proof of Lemma~\ref{vpronilpotent}, this implies that $\gamma_\infty(H)\leq Z(G')$.  Moreover  $G/\gamma_\infty(H)$ is virtually pronilpotent since the image of $H$  in the quotient is pronilpotent  and has finite index.  Thus,  in view of Lemma~\ref{vpronilpotent}, the quotient $\gamma_\infty(G)/\gamma_\infty(H)$ is finite-by-procyclic.  Using that $\gamma_\infty(H)\leq Z(G')$, as done above in Lemma~\ref{vpronilpotent}, we deduce that $\gamma_\infty(G)$ is finite-by-procyclic and the result holds.

Next we deal with the case where $G/F_2(G)$ is simple. Since $G$ is prosoluble, we have necessarily that the index  $j$ is a prime and  we can write $G=F_2(G)\langle a\rangle$, for some element $a\in G$. Since $F_2(G)$ is metapronilpotent, $\gamma_\infty(F_2(G))$ is finite-by-procyclic by Proposition~\ref{super}. In view of Remark~\ref{Re1} we can assume that  $\gamma_\infty(F_2(G))$ is procyclic. This implies that $G'\cap F_2(G)$ centralizes  $\gamma_\infty(F_2(G))$, that is  $\gamma_\infty(F_2(G)) \leq Z(G'\cap F_2(G))$. Hence  $G'\cap F_2(G)$ is central-by-pronilpotent and, in particular, pronilpotent. 
Observe that $[G,a]\leq G'\cap F_2(G)$ and so  $[G,a]\leq F(G)$. Since also $\gamma_\infty(F_2(G))\leq F(G)$, we deduce that  $\gamma_\infty(G)\leq F(G)$ and so it follows that  $G/F(G)$ is pronilpotent, that is, $G$ is metapronilpotent.  By  Proposition~\ref{super} the result follows.

\end{proof}

We will now prove Corollary \ref{cor}.
\begin{proof}[Proof of Corollary \ref{cor}]
In view of Theorem \ref{main} $\gamma_\infty(G)$ is finite-by-procyclic. By Lemma \ref{subgrchar} $\gamma_\infty(G)$ has a characteristic subgroup $N$ such that $\gamma_\infty(G)/N$ is procyclic. Let $C/N=C_{G/N}(\gamma_\infty(G)/N)$. Since the group of automorphisms of a procyclic group is abelian, it follows that $G/C$ is abelian. Moreover, by Lemma \ref{centralizer_in_theFitting} $C/N$ is pronilpotent, whence the result.  
\end{proof}

\section*{Declarations} 
\noindent {\bf Conflict of interest} The authors have no relevant financial and non-financial interests to disclose.


\begin{thebibliography}{40}
\bibitem{acc23}  Acciarri, C.,  Shumyatsky, P.: Varieties of groups and the problem on conciseness of words. ANNALI SCUOLA NORMALE SUPERIORE -- CLASSE DI SCIENZE, 23. \url{https://doi.org/10.2422/2036-2145.202306_017} (2024).
\bibitem{AGS_2023} Acciarri, C., Guralnick, R. M., Shumyatsky, P.: Criteria for solubility and nilpotency of finite groups with automorphisms, Bull. London Math. Soc. 55, 1340–1346 (2023).
\bibitem{AS_2022} Acciarri, C., Shumyatsky, P. Profinite groups with restricted centralizers of $\pi$-elements. Math. Z. 301, 1039–1045 (2022).
\bibitem{AS_2019} Acciarri C.,  Shumyatsky, P.: On groups in which Engel sinks are cyclic. J. Algebra 539, 366-376 (2019).
\bibitem{AS_2017} Acciarri C.,  Shumyatsky, P.: Commutators and commutator subgroups in profinite groups. J. Algebra 473, 166-182 (2017).
\bibitem{AS_2017_2}C. Acciarri, P. Shumyatsky, Coverings of commutators in profinite groups. Rend. Sem. Mat. Univ. Padova 137, 237--257 (2017).
\bibitem{AST_2014}Acciarri, C., Shumyatsky, P., Thillaisundaram, A.: Conciseness of coprime commutators in finite
groups. Bull. Aust. Math. Soc. 89, 252–258 (2014).
\bibitem{AS_2014}Acciarri, C., Shumyatsky, P.: On finite groups in which coprime commutators are covered by few cyclic subgroups, J. Algebra 407, 358--371 (2014).
\bibitem{AcSh14}  Acciarri, C., Shumyatsky, P.: On words that are concise in residually finite groups. J. Pure Appl. Algebra. 218, 130--134 (2014). 
\bibitem{AS_13}Acciarri C., Shumyatsky, P.:  On profinite groups in which commutators are covered by finitely many subgroups, Math. Z., 274, 239--248 (2013).
\bibitem{HPS23}  de las Heras, I., Pintonello, M., Shumyatsky, P.:  Strong conciseness of coprime commutators in profinite groups. J. Algebra. 633, 1--19   (2023).
\bibitem{DeMoSh23Isr}  Detomi, E., Morigi, M., Shumyatsky, P.:  Bounding the order of a verbal subgroup in a residually finite group. Isr. J. Math. 253, 771--785 (2023).
\bibitem{dks20} Detomi, E., Klopsch, B., Shumyatsky, P.:  Strong conciseness in profinite groups. J. Lond. Math. Soc. 102(3), 977--993 (2020). 
\bibitem{DMS_2021} Detomi, E., Morigi, M., Shumyatsky, P.: Strong conciseness of coprime and anti‑coprime commutators. Annali di Matematica 200, 945--952 (2021).
\bibitem{DeMoShEng1}  Detomi, E., Morigi, M., Shumyatsky, P.:  Words of Engel type are concise in residually finite groups. Bulletin of Mathematical Sciences. 9, 1950012  (2019). 
\bibitem{FAMS}Fern\'andez-Alcober, G.A.,  Morigi, M., Shumyatsky, P.: Procyclic coverings of commutators in profinite groups. Archiv Der Mathematik 2014, 103(2), 101--109 (2014).
\bibitem{Gorenstein} Gorenstein,D.: \textit{Finite groups}, Chelsea, New York, 1980.
\bibitem{KhSh23}  Khukhro, E., Shumyatsky, P.:   Strong conciseness of Engel words in profinite groups. Mathematische Nachrichten. 296(6), 2404--2416 (2023).
\bibitem{Isaacs} Isaacs,  I.\,M.: \textit{Finite group theory},  Graduate Studies in Mathematics \textbf{92}, Amer. Math. Soc., 2008.
\bibitem{Kelley}Kelley, J.L.: \emph{General Topology}. Van Nostrand, Toronto, 1955.
\bibitem{pishu}  Pintonello, M., Shumyatsky, P.:  On conciseness of the word in Olshanskii’s example. Arch. Math. 122, 241--247 (2024) \url{https://doi.org/10.1007/s00013-023-01955-x}
\bibitem{rei}  Reid,C.\,D.: Local Sylow theory of totally disconnected, locally compact groups, \textit{J.~Group Theory} \textbf{16}, 535--555  (2013).
\bibitem{rib-zal} Ribes, L.,  Zalesskii, P.: \emph{Profinite groups}, Springer, Berlin, 2010.
\bibitem{Robinson}Robinson, D.J.S.: \emph{Finiteness Conditions and Generalized Soluble Groups}. Part 1. Springer, New York, 1972.
\bibitem{Robinson_2}Robinson, D.J.S.: \emph{A Course in the Theory of Groups}. Graduate Texts in Mathematics, vol. 80, 2nd edn. Springer, New York (1996).
\bibitem{Shu_16} Shumyatsky, P.: On profinite groups with commutators covered by nilpotent subgroups. Rev. Mat. Iberoam. 32, no. 4, 1331--1339  (2016).
\bibitem{Pavel1} Shumyatsky, P.: Commutators of elements of coprime orders in finite group, \textit{Forum Math.} \textbf{27}, 575--583 (2015).
\bibitem{wil} Wilson,  J.\,S.:\emph{Profinite groups}, Clarendon Press, Oxford, 1998.
\bibitem{Zelmanov} Zelmanov, E.\,I.: On periodic compact groups, Israel J. Math. 77, 83--95  (1992).

\end{thebibliography}
\end{document}